\documentclass[11pt]{amsart} 
\usepackage{amssymb,amscd,amsthm,amsmath,amsfonts}
\usepackage[mathscr]{eucal} 

\newcommand {\norm}[1]{\mbox{$\left\|#1\right\|$}} 
\newcommand {\x}{\times} 
\newcommand {\cs}{\mbox{$C^{*}$-algebra}} 
\newcommand {\css}{\mbox{$C^{*}$-algebras}}

\newcommand {\ov}[1]{\mbox{$\overline{#1}$}}
\newcommand {\wht}[1]{\mbox{$\widehat{#1}$}}
\newcommand {\wte}[1]{\mbox{$\widetilde{#1}$}}

\newcommand {\C}{\mathbb{C}}

\newcommand {\R}{\mathbb{R}} 
 
\newcommand {\T}{\mathbb{T}} 
\newcommand {\Z}{\mathbb{Z}}

\newcommand {\al}{\mbox{$\alpha$}} 
 
\newcommand {\bt}{\mbox{$\beta$}} 
\newcommand {\ga}{\mbox{$\gamma$}} 
\newcommand {\Ga}{\mbox{$\Gamma$}} 
\newcommand {\de}{\mbox{$\delta$}} 
\newcommand {\De}{\mbox{$\Delta$}} 
\newcommand {\ze}{\mbox{$\zeta$}} 
 
\newcommand {\la}{\mbox{$\lambda$}}

\newcommand {\si}{\mbox{$\sigma$}}

\newcommand {\mcA}{\mathcal{A}}

\newcommand {\mcC}{\mathcal{C}}
\newcommand {\mcD}{\mathcal{D}}
\newcommand {\mcE}{\mathcal{E}}
\newcommand {\mcG}{\mathcal{G}}
\newcommand {\mcH}{\mathcal{H}}
\newcommand {\mcK}{\mathcal{K}}

\newcommand {\alrg}{A\rtimes _{\alpha,r} G}

\newcommand {\bgc}{\begin{center}}

\newcommand {\edc}{\end{center}} 
\newcommand {\be}{\begin{enumerate}} 
\newcommand {\ee}{\end{enumerate}} 
\newcommand {\beqn}{\begin{eqnarray}} 
\newcommand {\eeqn}{\end{eqnarray}} 
\newcommand {\beqns}{\begin{eqnarray*}} 
\newcommand {\eeqns}{\end{eqnarray*}} 
\newcommand {\bq}{\begin{quote}} 
\newcommand {\eq}{\end{quote}} 
 
\newcommand {\bi}{\begin{itemize}} 
\newcommand {\ei}{\end{itemize}} 
\newcommand {\bd}{\begin{description}} 
\newcommand {\ed}{\end{description}} 
 
\newcommand {\lan}{\mbox{$\langle$}} 
\newcommand {\ran}{\mbox{$\rangle$}}

\theoremstyle{plain} 
\newtheorem{theorem}{Theorem}

\newtheorem{proposition}{Proposition}

\numberwithin{equation}{section}

\begin{document}
\title{Contractive spectral triples for crossed products}
\author{Alan L. T. Paterson}
\address{University of Colorado, Department of Mathematics \\
c. o. 3709 Bluefield Court\\
Clarksville, Tennessee 37040-5597\\
USA}

\email{apat1erson@gmail.com}
\keywords{spectral triple, pointwise bounded, isometric, coactions} 
\subjclass{Primary 58B34, 46L55} 
\date{April, 2012}

\begin{abstract}
Connes showed that spectral triples encode (noncommutative) metric information.    Further, Connes and Moscovici in their metric bundle construction showed that, as with the Takesaki duality theorem, forming a crossed product spectral triple can substantially simplify the structure.     In a recent paper, Bellissard, Marcolli and Reihani (among other things) studied in depth metric notions for spectral triples and crossed product spectral triples for $Z$-actions, with applications in number theory and coding theory.    In the work of Connes and Moscovici, crossed products involving groups of diffeomorphisms and even of \'{e}tale groupoids are required.    With this motivation, the present paper develops part of the Bellissard-Marcolli-Reihani theory for a general discrete group action, and in particular, introduces coaction spectral triples and their associated metric notions.  The isometric condition is replaced by the contractive condition.
\end{abstract}
\maketitle

\section{Introduction}

Throughout the paper, $X=(A,\mathcal{H},D)$ will be a spectral triple in the sense of Connes (\cite{Connescompact,Connesbook}).
This can be defined as follows.  First, $A$ is a $\cs$, which we will always assume to be unital, equipped with a faithful (non-degenerate) representation $\pi$ on a Hilbert space $\mcH$, and second, $D$ is a (usually unbounded) self-adjoint operator on $\mcH$ with compact resolvent.  Third, we require that the set 
$\mcC^{1}(X)$ of $a$'s in $A$ for which $\pi(a)\,Dom(D)\subset Dom(D)$\footnote{As commented by Kaad and Lesch 
(\cite[Convention 4.2]{KaadLeschspec}), this assumption is very important, but is often slightly obscured in the literature.}   
and $\norm{[D,\pi(a)]}<\infty$ is dense in $A$.  (The operator $[D,\pi(a)]$ is at this stage, of course, only defined on 
$Dom \,D$, but since the latter is dense in $\mcH$ and  $[D,\pi(a)]$ is bounded, it extends by continuity to an element of $B(\mcH)$  with the same norm, and so can be regarded as actually belonging to $B(\mcH)$.)  We will sometimes regard $A$ as a subalgebra of $B(\mcH)$ and omit reference to the $\pi$.

In his development of noncommutative geometry, Connes showed that spectral triples not only give a context for K-homology and cyclic cohomology but also encode (noncommutative) {\em metric} information.  Particularly notable was his observation 
(e.g. \cite[VI.1]{Connesbook}) that for a compact spin manifold $M$, one can recover, among other things, the (geodesic) distance 
$d$ on $M$ from the canonical spectral triple $(C(M),\mcH,D)$
where (\cite[\S 5]{LawMich},\cite[II.7]{Shanahan}) $\mcH$ is the Hilbert space of $L^{2}$-spinors on $M$ 
and $D$ is the (self-adjoint) Dirac operator of $M$. This recovery is achieved by considering the space of Lipschitz functions $\mcA$ on $M$.  Indeed, each $a\in \mcA$ can be regarded as a multiplication operator on $\mcH$, and the commutator $[D,a]$ is densely defined and extends to a bounded linear operator on $\mcH$.  The distance function $d$ on $M$ is then determined for $p,q\in M$ by:
\begin{equation}  \label{eq:rpq}
d(p,q)=\sup\{\left| a(p) - a(q)\right|: \norm{[D,a]}\leq 1\}.
\end{equation} 
In particular, the right-hand side of (\ref{eq:rpq}) determines a metric for the topology of $M$.  We can, of course, think of points of $M$ as states on the $\cs$ $C(M)$, and Connes pointed out that, more generally, if we replace $a(p) - a(q)$ by 
$\phi(a) - \psi(a)$ above, we can extend the metric $d$ to a metric (also denoted $d$) on the state space $S(C(M))$ (i.e. the set of probability measures on $M$) of $C(M)$.  Further, the metric topology of $d$ on the state space is just the weak$^{*}$-topology.  This approach is motivation for replacing the special spectral triple $(C(M),\mcH,D)$ by an arbitrary spectral triple $X=(A,\mcH,D)$, and this gives a pseudo-metric $d_{X}$, or simply $d$, on $S(A)$.    Following \cite{BMR}, we will refer to $d$ as the {\em Connes pseudo-metric}.    So for $\phi, \psi\in S(A)$,
\begin{equation}  \label{eq:dphipsi}
d(\phi,\psi)=\sup\{\left|\phi(a) - \psi(a)\right|: 
a\in \mcC^{1}(X), \norm{[D,a]}\leq 1\}.
\end{equation}

Two natural questions arise.  (See the discussion in \cite[1.1]{BMR}.)  First, {\em when is $d$ actually a metric on 
$S(A)$?}  Referring to (\ref{eq:dphipsi}), we see that obstacles to this are (1) the degeneracy of the representation $\pi$ of $A$ on $\mcH$, and (2) there are non-trivial $a$'s (i.e. 
$a$'s that are not multiples of the identity) in the {\em metric commutant} 
($\{a\in \mcC^{1}(X): [D,a]=0\}$) of $D$.  In fact (\cite{Pav,Rieffel98,RenVar}) non-degeneracy for $\pi$ and triviality of the metric commutant are necessary and sufficient conditions for $d$ to be a metric.   The second question was raised and studied by Rieffel: {\em given that $d$ is a metric on $S(A)$, when does its metric topology coincide with the weak$^{*}$-topology?}    The answer in the unital case 
(\cite{Pav,Rieffel98,Rieffel99,OzRieff}) is that the two topologies coincide if and only if the image of the Lipschitz ball has compact closure in 
$A/\C 1$.  The corresponding result for the non-unital case was given by Latr\'{e}moli\`{e}re (\cite{Latre}).  

The main inspiration for the present paper is the recent work on spectral triples for group actions on the $\cs$ $A$ by Bellissard, Marcolli and Reihani (\cite{BMR}), in particular in the case when the group is $\Z$ (so that only a single automorphism of $A$ is involved).  For an ordinary metric space, there are a number of geometric notions associated with an action of a group on the space by homeomorphisms.  These include, in particular, the familiar notions of 
{\em quasi-isometric}, {\em equicontinuous} and {\em isometric}.  It is shown in \cite{BMR} that there are corresponding notions in the noncommutative case, i.e. for spectral triples.  (The definitions are given in \S 4 of the present paper.)

These noncommutative versions are used in a central theme of the investigations of \cite{BMR}, viz. given a spectral triple $X=(A,\mcH,D)$ where $A$ supports an action $\al$ of $\Z$ by automorphisms, how to define a {\em dual} spectral triple $Y$ on the (reduced) crossed product $\cs$ $A\rtimes_{\al,r} \Z$.    (In 
\cite[p.16]{BMR}, the authors write $Y=X\rtimes_{\al} \Z$ and call it the {\em regular representation of the metric dynamical system $(X,\al)$}.)   Motivation for such a study is that taking an appropriate dual action can greatly simply the study of the original spectral triple.   A remarkable example of how an appropriate crossed product can simplify the study of the original is in the von Neumann algebra category, where the Takesaki duality theorem (e.g. \cite[Theorem 13.3.7]{KR2}) says (among other things) that taking the crossed product of a von Neumann algebra for the action of the modular automorphism group (corresponding to a faithful normal state) transforms a type $III$ factor into a type $II_{\infty}$ von Neumann algebra.   (The $C^{*}$-algebra version of this is given by the Imai-Takai duality theorem  (\cite[Theorem 3.6]{ImaiTakai}, \cite[Theorem 3]{Landstad}, \cite[Theorem A.68]{EKQR}) which in its general form, uses the dual coaction - in particular, $G$ does not have to be abelian.)    A philosophically similar, but geometrical, situation arose in the work of Connes and Moscovici (\cite{ConnesMosc}) in the context of diffeomorphism invariant geometry.  There, one needs to consider the crossed product $C_{0}(W)\rtimes \Ga$ where $W$ is a compact Riemannian manifold and $\Ga$ a subgroup of $\mbox{Diff}(W)$.   In general, the action preserves no structure at all, in particular, no Riemannian metric is invariant under the action.  However, if we replace $W$ by the {\em metric bundle} $\mathcal{W}$ over $W$, whose fiber over $w\in W$ is the space of Euclidean metrics on the tangent space $T_{x}W$
then there is an invariant metric on $\mathcal{W}$ invariant under the natural action of $\Ga$, and the shift from $W$ to 
$\mathcal{W}$ corresponds to the shift from the type $III$ situation to one of type $II$ as above.  (See \cite[4.2]{BMR} for a detailed description of the construction of the metric bundle.)

Among a number of results in \cite{BMR}, the authors show the following (for a $\Z$-action $\al$ on $A$).  Given that $X$ is equicontinuous, there 
exists a natural ``dual'' spectral triple $Y$ for the reduced crossed product $A\rtimes_{\alpha,r} \mathbb{Z}$, where 
\[    Y=(A\rtimes_{\alpha,r} \mathbb{Z}, \mcK\otimes \mathbb{C}^{2}, \widehat{D}).  \]
Here, $\mcK$ is the space of sequences 
$\ell^{2}(\mathbb{Z},\mcH)=\mathcal{H}\otimes \ell^{2}(\mathbb{Z})$, and 
$\wht{D}$ is given by a diagonal operator whose entry over $n\in \Z$ is the $2\x 2$ matrix with zero diagonal entries and off-diagonal entries $D\mp \imath n$.   One considers  the dual action of $\wht{\Z}=\T$ on $A\rtimes_{\alpha,r} \mathbb{Z}$.   Further results in \cite{BMR} are:
\be
\item[(1)] $Y$ is isometric;
\item[(2)] if $X$ is such that the metric commutant is trivial and the image of the Lipschitz ball has compact closure in $A/\C 1$, then the Connes metrics induced on the state space of $A$ by both $X, Y$ are equivalent (and give the weak $^{*}$ topology of $A$);
\item[(3)] if $X$ is not equicontinuous but is quasi-isometric, it can effectively be replaced by a spectral triple that {\em is} equicontinuous (using a ``metric bundle'' construction inspired by that of Connes-Moscovici above).
\ee 
A number of interesting examples illustrating the theory is given.

These results involve, of course, actions by the group $\Z$.  However, it is desirable to extend them to actions by general discrete groups.   We saw this above in the discussion of the metric bundle, where the group acting could be any subgroup of 
$\mbox{Diff}(W)$. More generally, in further work of Connes and Moscovici (\cite{ConnesMoscHopf}), allowing for local rather than just global diffeomorphisms, one needs to consider the case where the transformation group is replaced by an \'{e}tale groupoid.  In this paper we will prove the general version of (1) for a discrete group acting on $A$.   While there is, of course, much more to be done to extend to this general context the other results in \cite{BMR}, even in the case of (1) alone, there are, as we shall see, questions that first have to be resolved.
  
We now define the two geometrical notions that we will require for an action $\al$ of a discrete group on a spectral triple 
$X=(A,\mcH,D)$.  First, we say that $X$ is {\em pointwise bounded} if the set
\begin{multline*}  
\mcC^{1}_{b}(G,X)=\{a\in \mcC^{1}(X): \al_{g}(a)\in \mcC^{1}(X) \mbox{ for all 
$g\in G$ and }\\
\sup_{g\in G}\norm{[D,\al_{g}(a)]}<\infty\}        
\end{multline*}
is dense in $A$ (or equivalently dense in $\mcC^{1}(X)$).  This is weaker than the ``equicontinuous'' condition used in \cite{BMR}: there, $X$ is equicontinuous if 
$\mcC^{1}_{b}(G,X)\\=\mcC^{1}(X)$.  The motivation for the terminology ``pointwise bounded'' is that for each appropriate ``point'' $a\in A$, the maps 
$g \to [D,\al_{g}(a)]$ are uniformly bounded, so that the set of ``functions'' $a\to [D,\al_{g}(a)]$ $(g\in G)$ is pointwise bounded.
We can think of this condition as corresponding  to the ``pointwise bounded'' condition in the classical 
Arzel\`{e}la-Ascoli theorem (cf. the use of equicontinuity in the  noncommutative Arzel\`{a}-Ascoli theorem, \cite[Theorem 1]{BMR}).
Pointwise boundedness is a natural condition to require.    (Indeed, the density of 
$\mcC^{1}(X)$ in $A$ in the spectral triple definition is just pointwise boundedness for the trivial group action.)   It is surely a weaker condition than equicontinuity, but unfortunately I do not have an example where pointwise boundedness holds but equicontinuity does not.   As in \cite[Definition 3]{BMR}, $X$ will be called  {\em isometric} if $\mcC^{1}(X)=\mcC^{1}_{b}(G,X)$ and 
\[         \norm{[D,\al_{g}(a)]}=\norm{[D,a]}\]
for all $a\in \mcC^{1}(X), g\in G$. 

One problem that arises when trying to prove a version of (1) for a general discrete group acting on $A$ is how to define $\wht{D}$.  What should we put in place of the $\mp n$?  However, $n$ can be recognized as coming from the usual word metric on the group $\Z$, and so for a general finitely generated, infinite group $G$, we need to replace $n$ by $c(g)$ where $c$ is the word metric on $G$ associated with a symmetric generating subset of 
$G$.  In fact, such a word metric is naturally associated (\cite{Connescompact}) with a spectral triple 
$(C_{r}^{*}(G),\ell^{2}(G),M_{c})$, where $M_{c}$ is the multiplication operator by $c$ on $\ell^{2}(G)$, $C_{r}^{*}(G)$ is the reduced $\cs$ of $G$, and $\wht{D}$ gives the unbounded Fredholm operator determining the Kasparov product of K-homology classes in the unbounded Fredholm picture 
(\cite{BaajJulg,Kucerovsky}, \cite[IV, Appendix A]{Connesbook}). 

A second problem is that in the 
$\mathbb{Z}$ case, the group was abelian, and we had the dual group $\mathbb{T}$ available to act on the crossed product.   This is no longer the case for general $G$.  Instead, as in the Imai-Takai duality theorem, we have to consider the dual coaction 
\[     \wht{\al}:A\rtimes_{\alpha,r} G\to 
(A\rtimes_{\alpha,r} G)\otimes C_{r}^{*}(G)        \]
where for $F\in C_{c}(G,A)$, 
\begin{equation}       \label{eq:dual}
 \wht{\al}(F)=\int \widetilde{\pi}(F(s))\widetilde{\lambda}_{s}\otimes \lambda_{s}.    
\end{equation}
Here, $\wht{\al}$ is continuous for the $A\rtimes_{\alpha,r} G$  norm restricted to $C_{c}(G,A)$, and so extends by continuity to the whole of $A\rtimes_{\alpha,r} G$.   Since, in the situation of the paper, $G$ will be discrete, the integrals involved are just summations, but we will stay with the familiar integral notations.  (This may also prove useful if the result of this paper can be extended to general $G$.)  Further, 
$(\widetilde{\pi},\widetilde{\lambda})$ is the covariant representation giving the regular representation of 
$A\rtimes_{\alpha,r} G$ as realized on 
$\mathcal{H}\otimes \ell^{2}(G)=\ell^{2}(G,\mcH)$, and $\lambda$ is the left regular representation of $G$ on $\ell^{2}(G)$.   As we will see, the dual spectral triple associated with $X=(A,\mcH,D)$ will be the triple $Y=(A\rtimes_{\alpha,r} G,\mathcal{H}\otimes \ell^{2}(G)\otimes \C^{2},\widehat{D})$.      In the abelian case, the dual coaction reduces to the familiar action of the dual group on the crossed product. 

We need geometric definitions of  metric notions (such as ``isometry'') for coactions just as we had for actions.   It is not immediately clear what they should be.   However, roughly, it is reasonable to think that if we dualize the coaction in some sense, then we should have something like an action (though not necessarily of a group).   More precisely, let $P_{r}(G)$ be the state space of $C_{r}^{*}(G)$.  This is a subsemigroup of 
$(C_{r}^{*}(G))^{*}$, which is an ideal in the Fourier-Stieltjes algebra $B(G)$ of $G$.  (The multiplication is just pointwise multiplication on $G$ when we regard the elements $\phi$ of $P_{r}(G)$ as functions on $G$ by setting $\phi(s)=\phi(\la_{s})$.)     It is this semigroup that we want acting in place of the group $G$.   For a general $\cs$ $B$ with a coaction $\de:B\to B\otimes C_{r}^{*}(G)$, the action $\bt$ of $P_{r}(G)$ on $B$ is given by slicing by 
$\phi\in P_{r}(G)$: so for $b\in B$, $\bt_{\phi}(b)=S_{\phi}(\de(b))$.

When $G$ is abelian, so that we are dealing with the dual action in place of the dual coaction,  the dual action on the crossed product is just $\bt$ restricted to the characters of $G$, the set of extreme points of $P_{r}(G)$.  The isometric condition makes good sense in this case since under the dual action, each character acts by multiplication as a unitary on $\ell^{2}(G)$.   However, when we extend this action to convex combinations in $P_{r}(G)$ of these characters, this is no longer the case.   Instead we need to replace the isometric condition by the {\em contractive} one.  In the case of a spectral triple $Y=(B,\mcK,D')$ - and we have in mind primarily the dual spectral triple - for which a coaction on $B$ is given, the idea is that $Y$ is {\em contractive} if 
the set
\[    \{b\in \mcC^{1}(Y): \norm{[\wht{D},\bt_{\phi}(b)]}\leq \norm{[\wht{D},b]}<\infty \mbox{ for all } \phi\in P_{r}(G)\}    \]
is dense in $B$ (or equivalently dense in $\mcC^{1}(Y)$).   (The precise definition is given in Section 4.)

The main result of this paper (Theorem~\ref{th:main}) (roughly) states that if $X$ is a pointwise bounded spectral triple, then the dual spectral triple 
$Y$ is contractive for the dual coaction.

As is well-known, some care has to be exercised with unbounded operators on a Hilbert space because of their partially defined domains.     The details on unbounded operators that we need for this paper are contained in the Appendix.    In particular, it gives a proof that the operator $\wht{D}$ used in the dual spectral triple really is a self-adjoint operator with compact resolvent.   (The author has been unable to find a written proof in the literature of this fundamental fact, and the proof also gives information about cores for  $\wht{D}$ that is needed in the proof of the main result of the paper.)    Also, the coaction literature can be rather technical, and for the reader who, like the present author, feels his or her background in the subject limited, I have tried to incorporate into the paper a simple account of the material, as self-contained as possible, that we need from the theory of reduced crossed products and coactions.   In particular, substantial simplifications result because in this paper, we only deal with the reduced case, the  group $G$ is discrete and the $\css$ involved unital.     (For a short, informative exposition (with proofs) of the general theory for full and reduced crossed products and coactions,  Appendix A of the memoir \cite{EKQR} by Echterhoff, Kaliszewski, Quigg and Raeburn is recommended.)

I am grateful to Kamran Reihani for helpful conversations.

\section{Preliminaries}

Let $G$ be a discrete group.  A {\em length} function on $G$ is a function 
$c:G\to \R$ satisfying: for every $s\in G$,
\begin{equation}   \label{eq:csbdd} 
\sup_{t\in G}\left| c(t) - c(st)\right|<\infty
\end{equation}
and $\left|c(t)\right|\to \infty$ as $t\to \infty$.  In particular, it is assumed that $G$ is countably infinite.     

The most important example of a length function is that of the {\em word length function} on $G$ (e.g. \cite[p.89]{Gromov}, \cite{Connescompact}).   Suppose that $G$ is infinite and finitely generated, and let $S$ be a finite, symmetric set of generators for $G$.  For $t\in G$, let $c(t)$ be the word norm associated with 
$S$, i.e. $c(t)$ is smallest integer $n$ such that $t$ can be written as a product of $n$ elements of $S$.   Then (as is easy to check) $c$ is indeed a length function.

We now establish notation for reduced crossed products for a locally compact group $G$.  (In our case, of course, $G$ is discrete.)  Let $A$ be a unitary 
$\cs$ and $(A,G,\al)$ be a dynamical system; so $\al:G\to Aut\, A$ is a homomorphism which is pointwise norm continuous.  Then (e.g.
\cite[7.6, 7.7]{Pederson}) $C_{c}(G,A)$ is a convolution normed algebra under the $L^{1}$-norm, and with product and involution given by:
\[   f*g(t)=\int f(s)\al_{s}(g(s^{-1}t))\,d\la(s)\hspace{.2in}
f^{*}(t)=\De(t)^{-1}\al_{t}(f(t^{-1})^{*}).   \]
The completion $L^{1}(G,A)$ of $C_{c}(G,A)$ is then a Banach algebra, and 
the full crossed product $A\rtimes_{\al} G$ is defined to be the enveloping $\cs$ of $L^{1}(G,A)$.  The (non-degenerate) representations of $A\rtimes_{\al} G$ are determined by the {\em covariant representations} $(\pi,u)$ on a Hilbert space $\mcK$ of 
$(A,G,\al)$, i.e. a pair for which $\pi$ is a representation of $A$ and $u$ a unitary representation of $G$ on the same Hilbert space $\mcK$ and for which 
$\pi(\al_{t}(a))=u_{t}\pi(a)u_{t}^{*}$ for all $a\in A, t\in G$.  
Such a covariant representation determines the corresponding representation 
$\pi\x u$ of $A\rtimes_{\al} G$ by defining
\begin{equation}   \label{eq:pixu} 
\pi\x u(F)=\int \pi(F(s))u_{s}\,d\la(s).
\end{equation}

In the present day study of crossed products and coactions, it is, for categorical reasons, usually desirable to work in the full setting because of the good universal properties. (See  \cite[A.9]{EKQR} for a discussion of the pros and cons of using the full or reduced theories.)   However, since we are concerned in this paper with spectral triples and such a triple involves an explicit Hilbert space, we will work with the {\em reduced} crossed product (as was the case in the early work on the subject, e.g. \cite{ImaiTakai,Landstad}).

The reduced crossed product
of $G$ and $A$ will denoted by $A\rtimes_{\al,r} G$.   It is a homomorphic image of the full crossed product and can be constructed as follows.  Let 
$\pi:A\to B(\mathcal{H})$ be a faithful, non-degenerate representation of $A$ on a Hilbert space 
$\mathcal{H}$.   Then (e.g. \cite[7.7]{Pederson}) there are a representation 
$\wte{\pi}$ of $A$ on $L^{2}(G,\mathcal{H})$ and a homomorphism 
$\wte{\la}:G\to U(B(L^{2}(G,\mathcal{H})))$ defined by:
\[   \wte{\pi}(a)\xi(t)=\pi(\al_{t}^{-1}(a))\xi(t),\hspace{.2in} 
\wte{\la_{s}}\xi(t)=\xi(s^{-1}t)    \]
for $\xi\in L^{2}(G,\mathcal{H})$.    Let $\la$ be the left regular representation of $G$ on $L^{2}(G)$: $\la_{s}f(t)=f(s^{-1}t)$.)
The pair $(\wte{\pi},\wte{\la})$ is a covariant representation of $(A,G,\al)$ and hence determines a representation 
$\hat{\pi}=\wte{\pi}\x \wte{\la}$ of $A\rtimes_{\al} G$.  From (\ref{eq:pixu}), for $F\in C_{c}(G,A), \xi\in L^{2}(G,\mcH)$,
\begin{equation}   \label{eq:hatpi}
 \hat{\pi}(F)\xi(t)=\int \wte{\pi}(F(s))(\wte{\la_{s}}\xi)(t)=\int \pi(\al_{t^{-1}}(F(s)))\xi(s^{-1}t).  
 \end{equation}
The image of this representation is the reduced crossed product $A\rtimes_{\al,r} G$, realized spatially on $L^{2}(G,\mcH)$.  
As is customary, for notational simplicity,  we sometimes identify $F\in C_{c}(G,A)$ with its image $\wht{\pi}(F)$. 

If $G$ is abelian, then (e.g. \cite[pp.265-266]{Landstad}, \cite[A.3]{EKQR}) there is an action of the dual group $\wht{G}$ on $A\rtimes_{\al,r} G$ called the {\em dual action}.   This action is defined by: for $\ga\in \wht{G}$ and $F\in C_{c}(G,A)$, we take
$\wht{\al}_{\ga}F(s)=\ga(s)F(s)$.  The map $\wht{\al}_{\ga}$ extends by continuity to give an automorphism on 
$A\rtimes_{\al,r} G$, and $(A\rtimes_{\al,r} G,\wht{G},\wht{\al})$ is a $C^{*}$-dynamical system.   (Often, authors define the dual action using the complex conjugate of $\ga(s)$, i.e. $\wht{\al}F(s)=\ov{\ga(s)}F(s)$, and (cf. \cite[pp.26, 194-195]{Williams}) either choice is fine, depending on how we identify elements in the dual group with actual functions on the group.     However, in the study of coactions, it is more convenient to use the $\wht{\al}F(s)=\ga(s)F(s)$ version for the dual action.)

In the non-abelian case, the dual group is no longer relevant for duality purposes, and instead one replaces the dual action in the abelian case by the dual coaction.  The definition of coaction which we now give is for the case where $G$ is discrete and the $\cs$ $B$ unital, the general case being more involved (in particular, requiring the use of multiplier algebras).  So let $B$ be a unital $\cs$, and  $id_{B}$ be the identity map on $B$ and $id_{G}$ the identity map on 
$C_{r}^{*}(G)$.  Let $\de_{G}:C_{r}^{*}(G)\to C_{r}^{*}(G)\otimes C_{r}^{*}(G)$ be the homomorphism determined by: 
$\de_{G}(\la_{s})=\la_{s}\otimes \la_{s}$.  (This extends continuously to $C_{r}^{*}(G)$ since 
(e.g. \cite[13.11.3]{D2} or \cite[pp.131-132]{EKQR}) $\la\otimes\la$ is weakly contained in $\la$.)
A (reduced) {\em coaction} for $B$ (with respect to $G$) is a unital injective homomorphism $\de:B\to B\otimes C_{r}^{*}(G)$ that satisfies the coaction identity: 
\[    (\de\otimes id_{G})\circ \de = (id_{B}\otimes \de_{G})\circ \de.    \]
Of course, if (as will be in our case) $B$ is a $C^{*}$-subalgebra of $B(\mcK)$ ($\mcK$ a Hilbert space) then $B\otimes C_{r}^{*}(G)\subset 
B(\ell^{2}(G,\mcK))$ so that $\de$ will also be an injective homomorphism into $B(\ell^{2}(G,\mcK))$.  (Coactions are also required to be 
{\em non-degenerate} (e.g.\cite[p.256]{Landstad}) - this condition is always satisfied by the dual coaction, the only coaction with which we will be concerned in this paper, and so we will not define non-degeneracy here.)

Of particular importance is the {\em dual coaction} 
$\wht{\al}$ for $A\rtimes_{\al,r} G$, defined in (\ref{eq:dual}). It is easily checked that $\wht{\al}$ satisfies the coaction identity.  If $G$ is abelian, then
the dual action and the dual coaction are effectively the same, the relation between them being given by: 
$(1\otimes 1\otimes \si_{\chi})(\wht{\al}(F))=\al_{\chi}(F)$
where $\si_{\chi}$ is the state on $C_{r}^{*}(G)\cong C_{0}(\wht{G})$ associated with point evaluation at $\chi$: $\si_{\chi}(\la_{s})=\chi(s)$.  (We will return to this more generally in Section 3, and for this, as we will see, it is helpful to use {\em slice maps}  (below).) 

First, let $P_{r}(G)$ be the state space of 
$C_{r}^{*}(G)$.  Since $C_{r}^{*}(G)$ is unital, $P_{r}(G)$ is a weak$^{*}$ compact, convex subset of $B_{r}(G)=C_{r}^{*}(G)^{*}$.  (A brief discussion of $B_{r}(G)$ is given on \cite[p.258]{Landstad}.)  The canonical embedding of $B_{r}(G)$  into $B(G)=C^{*}(G)^{*}$ (itself coming from the canonical homomorphism from $C^{*}(G)$ onto $C_{r}^{*}(G)$) identifies the Banach space $B_{r}(G)$ with a subspace of $B(G)$, the Fourier-Stieltjes algebra of $G$, and $P_{r}(G)$ with a weak$^{*}$-compact convex subset of the state space $P(G)$ of $C^{*}(G)$.  Now regard $B(G)$ as a space of functions on $G$, where, for 
$\phi\in B(G)$, $\phi(s)=\phi(\la_{s}^{u})$, where $s\to \la_{s}^{u}$ is the canonical homomorphism from $G$ into the unitary group of $C^{*}(G)$.  Then 
$B_{r}(G)$ is a (normed closed) ideal in $B(G)$ and $P_{r}(G)$ is a subsemigroup of $P(G)$.    As a function on $G$, $\phi\in B_{r}(G)$ is given by:  
$\phi(s)=\phi(\la_{s})$, and since $\de_{G}(s)=\la_{s}\otimes \la_{s}$, the product on $B_{r}(G)$ can be defined by: $\phi\psi=(\phi\otimes \psi)\circ \de_{G}$.  
In order to associate an action of $B_{r}(G)$ - and hence of $P_{r}(G)$ - on 
$B$ for a coaction with respect to $G$, we use {\em slice maps} (e.g. \cite[Chapter 8]{Lance},  \cite[A.4]{EKQR}).    (As commented in \cite[A.4]{EKQR}, slicing in tensor products is one of the basic tools in the theory of coactions.)   

If $A_{1}, A_{2}$ are $\css$ realized on Hilbert spaces $\mcH_{1}, \mcH_{2}$, let $A_{1}\odot A_{2}$ be the span of simple tensors 
$a_{1}\otimes a_{2}$ in $B(\mcH_{1}\otimes \mcH_{2})$ ($a_{i}\in A_{i}$).  
The closure of $A_{1}\odot A_{2}$ is the spatial tensor product 
$A_{1}\otimes A_{2}$ of $A_{1}$ and $A_{2}$.  If $\phi\in A_{2}^{*}$, then the slice map $S_{\phi}:A_{1}\otimes A_{2}\to A_{1}$ is a well-defined bounded linear map of norm $\norm{\phi}$ and is determined by its value on $A_{1}\odot A_{2}$:
\[    S_{\phi}(a_{1}\odot a_{2})=a_{1}\phi(a_{2}).       \]

Next for $c\in A_{1}\otimes A_{2}$, the map $\phi\to S_{\phi}(c)$ is weak$^{*}$-norm continuous on bounded subsets of $A_{2}^{*}$.   To show this,  let 
$\phi_{n}\to \phi$ weak$^{*}$ in a bounded subset of $A_{2}^{*}$.  Trivially, $S_{\phi_{n}}(a\otimes b)=\phi_{n}(b)a\to \phi(b)a=
\phi(a\otimes b)$ in norm for every simple tensor $a\otimes b$.  Hence this result is also true for elements in the span $C$ of such tensors in 
$A_{1}\otimes A_{2}$.  A ``uniform convergence'' type argument, using the density of $C$ in $A_{1}\otimes A_{2}$ and the norm boundedness of $\{S_{\phi_{n}}\}$, then gives the result.

Given a coaction $\de:B\to B\otimes C_{r}^{*}(G)$, the (left) action $\bt$ of 
$B_{r}(G)$ on $B$ is defined by: $\bt_{\phi}(b)=\phi.b$ where
\begin{equation}   \label{eq:btphi} 
\phi.b=S_{\phi}(\de(b)). 
\end{equation}  
To check that this is an action, we have to show (among other things) that 
$\phi.(\psi.b)=(\phi\psi).b$, i.e. 
$S_{\phi}(\de(S_{\psi}(\de(b))))=S_{\phi\psi}(\de (b))$.  This amounts to showing that for $b\in B$,
\[ S_{\phi}(\de(S_{\psi}(\de(b))))=
S_{\phi\otimes \psi}((1\otimes \de_{G})(\de(b)))
=S_{\phi\otimes \psi}((\de\otimes 1)(\de(b)))  \]
(using the coaction identity).  It is simple to prove this by approximating 
$\de(b)$ by a finite sum of simple tensors $\sum b_{i}\otimes s_{i}$ ($b_{i}\in B, s_{i}\in G$) then similarly, each $\de(b_{i})$ by a finite sum 
$\sum b_{ij}\otimes s_{ij}$.  The remaining verifications that $B$ is a left Banach $B_{r}(G)$-module are easy.

\section{Crossed product spectral triples}  \label{sec:crossed}

We saw above that every coaction on a $\cs$ $B$ gives rise to an action $\phi\to \bt_{\phi}$ of the semigroup $P=P_{r}(G)$ on $B$.  
In the case which concerns us in this paper, viz. where $\de$ is the dual coaction for an action $\al$ of $G$ on a 
$\cs$ $A$, the action $\bt$ of $B_{r}(G)$ on $B=\alrg\subset B(\mcH\otimes \ell^{2}(G)) $ is easy to calculate, and fits in well with the familiar dual action for the commutative case.  

Indeed (cf. \cite[Theorem 4]{Landstad}) for $F\in C_{c}(G,A)$, 
\begin{gather}   \label{eq:btphM}
 \bt_{\phi}F=S_{\phi}(\int \widetilde{\pi}(F(s))\widetilde{\lambda}_{s}\otimes \lambda_{s}) =
 \int \widetilde{\pi}(\phi(s)F(s))\widetilde{\lambda}_{s}
\end{gather}
so that $\bt_{\phi}F$ is just pointwise multiplication by $\phi$ on $C_{c}(G,A)$, exactly the same as what happens in the abelian case with the characters of $G$.  In that case, $C_{r}^{*}(G)=C_{0}(\wht{G})$, so that 
$B_{r}(G)=M(\wht{G})$, and $P$ is just the set of probability measures on $\wht{G}$.   The extreme points of $P$ are just the characters of $G$, and restricting the action of $P$ to these gives the dual action of $\wht{G}$ on
$\alrg$.  Of course, when $G$ is not abelian, the extreme points of $P$ are the pure states on $C_{r}^{*}(G)$ which are usually not characters.  There is no advantage in restricting the action of $P$ to the pure states, and by doing that, we also lose the semigroup structure of
 $P$.  For these reasons, for general $G$, we use the action of $P$ on $\alrg$.  

Now let $c$ be a length function on $G$, and $M_{c}$ be the multiplication operator 
by $c$ on $\ell^{2}(G)$: so $(M_{c}\xi)(t)=c(t)\xi(t)$ defined for the subspace 
$\mcD$ of elements $\xi\in \ell^{2}(G)$ for which 
$\sum_{t\in G} c(t)^{2}\left| \xi(t)\right|^{2}<\infty$.  Then (\cite[2.7.1]{KR1})
$M_{c}$ is an unbounded self-adjoint operator on $\ell^{2}(G)$ with domain 
$\mcD$.  Further, $C_{c}(G)\subset \ell^{2}(G)$ is a core for $M_{c}$.    Also, since each $c(t)$ is an integer and $\left|c(t)\right|\to \infty$, the operator 
$(M_{c} - \imath)^{-1}$ is compact, so that $M_{c}$ has compact resolvent.   Let 
$Z=(C_{r}^{*}(G),\ell^{2}(G),M_{c})$.   For each $s\in G$ let $m_{s}=\sup_{t\in G} \left| c(t) - c(s^{-1}t)\right| < \infty$ (by (\ref{eq:csbdd})) .
Let $B$ be the space of functions $f\in \ell^{1}(G)$ for which $mf\in \ell^{1}(G)$ where $(mf)(s)=m_{s}f(s)$, and let $\pi_{r}$ be the left regular representation of $C_{r}^{*}(G)$.    Then
$C_{c}(G)\subset B\hookrightarrow C_{r}^{*}(G)$, so that $B$ is a dense subspace of $C_{r}^{*}(G)$.   Then for $f\in B$ and $\xi\in \mcD$, 
$\left| [M_{c},\pi_{r}(f)]\xi(t)\right|=\left|\sum_{s\in G}[c(t) - c(s^{-1}t)]f(s)\xi(s^{-1}t)\right|
\leq \sum_{s\in G} m_{s}\left|f(s)\right|\left|\xi\right|(s^{-1}t) = (\left|mf\right| * \left|\xi\right|)(t). $
So 
\begin{equation}  \label{eq:Mclaf}
\norm{ [M_{c},\pi_{r}(f)]\xi}_{2}\leq \norm{mf}_{1}\norm{\xi}_{2}   
\end{equation} 
from which it follows that $f\in \mcC^{1}(Z)$ and that $Z$ is a spectral triple.  

So we now have two spectral triples $X=(A,\mcH,D)$ and $Z=(C_{r}^{*}(G),\ell^{2}(G),\\
M_{c})$.    In particular, both $D, M_{c}$ are self-adjoint operators with compact resolvent, and so by Proposition~\ref{prop:Dsa} with $\mcK=\mcH\otimes \ell^{2}(G)$, the operator $\wht{D}$ on $\mcH\otimes \ell^{2}(G)\otimes \C^{2}$, where
\begin{equation}   \label{eq:DiD}
  \widehat{D}=
\begin{bmatrix}
0 & \wht{D}_{-}\\
\wht{D}_{+} & 0
\end{bmatrix}
\end{equation}
is self-adjoint with compact resolvent, where  $\wht{D}_{\mp}=D\otimes 1 \mp \imath\otimes M_{c}$.    Again from Proposition~\ref{prop:Dsa}, 
the domain of $Dom\,\wht{D}_{\mp}$ is $\wht{V}$ and  $Dom\,\wht{D}=\wht{V}^{2}=\wht{V}\oplus \wht{V}$ where $\wht{V}$ is given in (\ref{eq:whtV}).  

In the next section, we will show that if $X$ is pointwise bounded (for the $G$-action), then the triple $(\alrg, \mcH\otimes \ell^{2}(G)\otimes \C^{2},\wht{D})$ is in fact a spectral triple, which we will call the {\em dual spectral triple} for $X$.   Further, $Y$ will be shown to be contractive for the dual coaction.

\section{The main result}

Let $X=(A,\mcH,D)$ be a spectral triple (defined in the Introduction to this paper).  It is obvious from the definition of $\mcC^{1}(X)$ and the Leibniz formula that $\mcC^{1}(X)$ is a subalgebra of $A$.  (In fact (\cite[(2.1), Lemma 1]{BMR}) $\mcC^{1}(X)$ is a Banach $^{*}$-algebra invariant under the holomorphic functional calculus where the Banach algebra norm is given by: $\norm{a}_{1}=\norm{a} + \norm{[D,a]}$.)

Let $G$ be a locally compact group and $\al$ as above be an action of $G$ on $A$ (i.e. a strongly continuous homomorphism $\al$ of $G$ into the $^{*}$-automorphism group $Aut(A)$ of $A$).  In \cite{BMR}, the authors define three noncommutative geometric properties with respect to the group action on $A$.  The names given are those used in the ``commutative'' case of a group action on a locally compact metric space.  Let $\mcC^{1}(G,X)$ be the set of $a\in A$ such that $\al_{t}(a)\in \mcC^{1}(X)$ for all $t\in G$ and the map $t\to [D,\al_{t}(a)]$ is norm continuous.   Note that taking $t=e$ in this definition gives that 
$\mcC^{1}(G,X)\subset \mcC^{1}(X)$.   Now define  $\mcC_{b}^{1}(G,X)$ to be the set of $a$'s in $\mcC^{1}(G,X)$ such that 
$\sup_{t}\norm{[D,\al_{t}(a)]}<\infty$.   Note that if $a\in \mcC_{b}^{1}(G,X)$ then so also is every $\al_{t}(a)$.
Then $X$ is called {\em quasi-isometric} if $\mcC^{1}(G,X)=\mcC^{1}(X)$.     The spectral triple $X$ is called {\em equicontinuous} if 
$\mcC^{1}_{b}(G,X)=\mcC^{1}(X)$.  Last, it is called {\em isometric} if it is quasi-isometric and for all 
$a\in \mcC^{1}(X)$ and all $t\in G$, we have $\norm{[D,\al_{t}(a)]}=\norm{[D,a]}$.  (In particular, $X$ is equicontinuous if it is isometric.)     The condition that we will be concerned with in this paper is similar to that of equicontinuity but not quite so strong.   We will call $X$  {\em pointwise bounded} if $\mcC^{1}_{b}(G,X)$ is dense in $\mcC^{1}(X)$ (and hence by the spectral triple requirement, dense in $A$.)
  
We now turn to the corresponding definitions for a {\em coaction} $\de:B\to B\otimes C_{r}^{*}(G)$ of a unital $\cs$ $B$ instead of an {\em action} of $G$ on $A$.   So let $Y=(B,\mcK,D')$ be a spectral triple.  Then ((\ref{eq:btphi})) associated with $\de$ is the semigroup action $\phi\to \bt_{\phi}$ of $P=P_{r}(G)$ on $B$.  Then, similar to the  definitions for an action, we define $\mcC^{1}(P,Y)$ to be the set of $b$'s in $\mcC^{1}(Y)$ such that for all $\phi\in P$, $\bt_{\phi}(b)\in \mcC^{1}(Y)$
and the map $\phi\to [D',\bt_{\phi}(b)]$ is weak$^{*}$-norm continuous.  Next $\mcC^{1}_{b}(P,Y)$ is defined to be the set of $b$'s in $\mcC^{1}(P,Y)$ such that 
$\sup_{\phi\in P}\norm{[D,\bt_{\phi}(b)]}<\infty$.   As in the group action case, we say that $Y$ is {\em quasi-isometric} if $\mcC^{1}(P,Y)=\mcC^{1}(Y)$.   The spectral triple $Y$ is called  {\em equicontinuous} if $\mcC^{1}_{b}(P,Y)=\mcC^{1}(Y)$.  We replace the isometric condition of the action case by the contractive condition: $Y$ is called {\em contractive} if 
\[ \mcC^{1}_{contr}(P,Y)=\{b\in B:  \norm{[D',\bt_{\phi}(b)]}\leq \norm{[\wht{D},b]}<\infty \mbox{ for all } \phi\in P_{r}(G)\}    \]
is dense in $B$.   (For justification of this definition (and as we will see later), for abelian discrete $G$, the isometric condition for the dual action is equivalent to the contractive condition for the dual coaction.)   Last, the spectral triple $Y$ is called {\em pointwise bounded} if $\mcC^{1}_{b}(P,Y)$ is dense in $\mcC^{1}(Y)$ (and hence dense in $B$). 

In this paper, we will only have occasion to use pointwise boundedness for group actions  and the contractive condition for coactions.   
As in the previous section, $X$ will be the spectral triple
$(A,\mcH,D)$ and $Y$ the triple\\  $(A\rtimes_{\al,r} G, \ell^{2}(G,\mcH)\otimes \C^{2},\wht{D})$. 

\begin{proposition}    \label{prop:Yst}
Suppose that $X$ is pointwise bounded.   Then $Y$ is a spectral triple.
\end{proposition}
\begin{proof}
We only need to show that  $\mcC^{1}(Y)$ is dense in $B$, since $Y$ satisfies all the other requirements for a spectral triple.  Since $X$ is pointwise bounded, 
$\mcC^{1}_{b}(G,X)$ is dense in $A$.  Now let $\mcC$ be the space of functions $F:G\to C^{1}_{b}(G,X)$ that vanish outside a finite subset of $G$.  It is obvious that $\mcC$ is dense in $\ell^{1}(G,A)$ and hence its image, also denoted $\mcC$, is dense in $A\rtimes_{\al,r} G$.  It is sufficient, then, to show that $\mcC\subset \mcC^{1}(Y)$.    Let $\mcE=Dom\, D\odot C_{c}(G)\subset \mcH\otimes \ell^{2}(G)$.     Note that since 
$G$ is discrete, $\mcE=C_{c}(G,Dom\,D)$ and that $\mcE$ is invariant under both $D\otimes 1$ and $1\otimes M_{c}$.   Since $Dom\,D$ and 
$C_{c}(G)$ are respectively cores for $D, M_{c}$, it follows by Proposition~\ref{prop:Dsa} that $\mcE^{2}$ is a core for $\wht{D}$.   Next we claim that 
for $F\in \mcC$, we have 
\begin{equation}   \label{eq:FEE}
\wht{\pi}(F)\mcE\subset \mcE.     
\end{equation}           
To see this, let $\xi\in \mcE$.   Then $\wht{\pi}(F)\xi(t)=\int \pi(\al_{t^{-1}}(F(s))\wte{\la}_{s}\xi(t)$.  Now for each $s$,  
 $\wte{\la}_{s}\xi\in \mcE$ and since each $F(s)\in \mcC^{1}_{b}(G,X)$,  so also does every $\al_{t^{-1}}(F(s))$, in particular, it belongs to $\mcC^{1}(X)$ and so preserves the domain of $D$.    So the map $F_{s}$ given by
$t\to \pi(\al_{t^{-1}}(F(s))\wte{\la}_{s}\xi(t)$ sends $G$ into $Dom\,D$.   Further, $F_{s}$ has finite support since $\xi$ has and so $F_{s}\in \mcE$.   
Since $F$ vanishes off a finite subset of $G$, $\wht{\pi}(F)\xi$ is a finite sum of $F_{s}$'s and so $\wht{\pi}(F)$ maps $\mcE$ into $\mcE$ giving (\ref{eq:FEE}).
It also follows that the commutators $[D\otimes 1,F], [1\otimes M_{c}]$ are operators on $\mcE$, and we now calculate them.   (Recall that, when convenient, we identify $F$ with $\wht{\pi}(F)$. )   

First we claim that for $\xi\in \mcE, t\in G$,
\begin{equation}  \label{eq:DpiF}
([D\otimes 1,F]\xi)(t)=\int [D,\pi(\al_{t^{-1}}(F(s))]\xi(s^{-1}t).
\end{equation}
For, recalling that $\wht{\pi}(F)$ is a finite sum of $F_{s}$'s,
\begin{align*}
\begin{split}
&[D\otimes 1,F]\xi(t) \\
&=((D\otimes 1)\wht{\pi}(F)\xi)(t) - (\wht{\pi}(F)(D\otimes 1)\xi)(t) \\
& =(D\otimes 1)\int \pi(\al_{t^{-1}}(F(s))\wte{\la}_{s}\xi(t) - \int \pi(\al_{t^{-1}}(F(s)))\wte{\la}_{s}((D\otimes 1)\xi)(t)\\
&= \int D(\pi(\al_{t^{-1}}(F(s))))\xi(s^{-1}(t) - \int \pi(\al_{t^{-1}}(F(s))D(\xi(s^{-1}t))\\
& =\int [D,\pi(\al_{t^{-1}}(F(s))]\xi(s^{-1}t).
\end{split}
\end{align*}

Next, we show that 
\begin{equation}  \label{eq:cpila}
[1\otimes M_{c},F]\xi(t)= \int \pi(\al_{t^{-1}}(F(s)))[c(t) - c(s^{-1}t)]\xi(s^{-1}t).
\end{equation}
For
\begin{align*}
[1\otimes M_{c},\wht{\pi}(F)]\xi(t) & = c(t)\int \pi(\al_{t^{-1}}(F(s)))\wte{\la}_{s}\xi(t) - \int \pi(\al_{t^{-1}}(F(s)))\wte{\la}_{s}(c\xi)(t)\\
& = \int \pi(\al_{t^{-1}}(F(s)))[c(t) - c(s^{-1}t)]\xi(s^{-1}t).
\end{align*}

We now want to show that each of the commutators in (\ref{eq:DpiF}), (\ref{eq:cpila}) is a bounded map on $\mcE$.   For the first of these, suppose that 
$a\in \mcC^{1}_{b}(G,X)$, $f\in 
C_{c}(G)$ and take $F=a\otimes f\in \mcC$.  Then
\begin{equation}   \label{eq:D1F}
\norm{[D\otimes 1,F]}
\leq (\sup_{t}\norm{[D,\al_{t^{-1}}(a)]})\norm{f}_{1}.
\end{equation}
For let $M=\sup_{t}\norm{[D,\al_{t^{-1}}(a)]}$ and $\xi,\eta\in \mcE$ .   Then $M<\infty$ since $a\in \mcC^{1}_{b}(G,X)$, and 
 by (\ref{eq:DpiF}),
\begin{align*}
\left|\lan [D\otimes 1, F]\xi,\eta\ran\right| & =\left|\iint \lan[D,\pi(\al_{t^{-1}}(f(s)a))]\xi(s^{-1}t),\eta(t)\ran \,ds\,dt\right| \\
& \leq \iint \left|\lan[D,\pi(\al_{t^{-1}}(a))]\xi(s^{-1}t),\eta(t)\ran\right| \left|f(s)\right|\,ds\,dt  \\
& \leq \iint \left|f(s)\right|M\norm{\xi(s^{-1}t)}\norm{\eta(t)}\,ds\,dt \\
&\leq M\int \left|f(s)\right|(\int \norm{\xi(s^{-1}t)}^{2}\,dt)^{1/2}(\int \norm{\eta(t)}^{2}\,dt)^{1/2}\,ds \\
& \leq M\norm{f}_{1}\norm{\xi}\norm{\eta}
\end{align*} 
giving (\ref{eq:D1F}).     

Since every $F\in \mcC$ is a linear combination of terms of the form $a\otimes f$, 
it follows that for general $F\in \mcC$,  
 $[D\otimes 1,F]$ is a bounded operator on $\mcE$.  
The boundedness of the second commutator  on $\mcE$ follows similarly using (\ref{eq:cpila}) since for each $s$,  
$\sup_{t}\left| c(t) - c(s^{-1}t)\right| <\infty$ by (\ref{eq:csbdd}) and $f(s)\ne 0$ for only finitely many $s$.   Precisely,
\[  \norm{[1\otimes M_{c},a\otimes f]}\leq \norm{a}\norm{f}_{1}\sup\{\left|c(t) - c(s^{-1}t\right|: f(s)\ne 0, t\in G\}.    \]
So for $F\in \mcC$, the commutators 
\[   [D\otimes 1 \pm \imath 1\otimes M_{c},F]= [D\otimes 1,F] \pm \imath [1\otimes M_{c},F]     \]
are bounded operators on $\mcE$.    Let $F'=F\oplus F$, a diagonal operator on $\ell^{2}(G,\mcH)\otimes \C^{2}$.   Then the commutator 
$[\wht{D},F']$ has zero diagonal and off-diagonal entries 
$[D\otimes 1 \pm \imath 1\otimes c,F]$, and is as well a bounded operator on $\mcE\oplus \mcE$.    Now apply (\ref{eq:FEE}) and 
Proposition~\ref{prop:coredom} to conclude that $F'\in \mcC^{1}(Y)$, and that $Y$ is a spectral triple.
 \end{proof}
 
 The next theorem is the main result of this paper.    (Note that if $G$ is discrete and finitely generated, then we can take $c$ to be any word  length function on 
 $G$.)
 
\begin{theorem}   \label{th:main}
Let $A$ be a unital $\cs$, $X=(A,\mcH, D)$ a spectral triple, $G$ a discrete countably infinite group with length function $c$ and 
$(A,G,\al)$ be a $C^{*}$-dynamical system.  Suppose that $X$ is pointwise bounded for $G$.   Then the dual spectral triple
$Y=(\alrg ,\ell^{2}(G,\mcH)\otimes \C^{2}, \wht{D})$ is contractive for the dual coaction 
$\de:\alrg\to (\alrg)\otimes C_{r}^{*}(G)$.
\end{theorem}
\begin{proof} 
By Proposition~\ref{prop:Yst}, $Y$ is a spectral triple.    It remains to show that $Y$ is contractive.    It is sufficient to show that
$\mcC=C_{c}(G,\mcC^{1}_{b}(G,X)$ (which we used in the previous proof) is a subspace of $\mcC^{1}_{contr}(P,Y)$.     First, $\mcC$  is $P$-invariant.   This is trivial, since if $F\in \mcC$, i.e. the map $F:G\to \mcC^{1}_{b}(G,X)$ has finite support,  so also does $\bt_{\phi}F$ (since $\bt_{\phi}F(s)=\phi(s)F(s)$).    It remains to show that 
\[  \norm{[D,\bt_{\phi}(F)]}\leq \norm{[D,F]}            \]
and that the map $\phi\to [D,\bt_{\phi}(F)]$ is weak$^{*}$-norm continuous.

To this end, define (\cite{ImaiTakai,Landstad}) the unitary $W$ on 
$\ell^{2}(G\x G,\mcH)$ by: 
\[  W\zeta(s,t)=\zeta(s,s^{-1}t).      \]
Then $W^{*}\zeta(s,t)=\zeta(s,st)$ and trivially, $W$ is unitary.  We shall also use $W$ for the case $\mcH=\C$.    For $t\in G$, $\ze\in \ell^{2}(G\x G,\mcH)$, 
$\ze_{t}\in \ell^{2}(G,\mcH)$ is given by: $\ze_{t}(s)=\ze(s,t)$.  

Let $\de$ be the dual coaction on 
$B$.  Then for $F\in C_{c}(G,A)$,
\begin{align*}
W(F\otimes 1)W^{*}\ze(v,t) & =(F\otimes 1)W^{*}\ze(v,v^{-1}t) = F(( W^{*}\ze)_{v^{-1}t})(v)  \\
&= \int \pi(\al_{v^{-1}}(F(s)) (W^{*}\ze)_{v^{-1}t}(s^{-1}v)\,ds  \\
&=\int \pi(\al_{v^{-1}}(F(s))(W^{*}\ze)(s^{-1}v,v^{-1}t)\,ds  \\
&=\int \pi(\al_{v^{-1}}(F(s))\ze(s^{-1}v,s^{-1}t)\,ds =\de(F) 
\end{align*}
using the formula (\ref{eq:hatpi}).   It follows by continuity that for $w\in \alrg$,
\begin{equation}   \label{eq:pilade}
W(w\otimes 1)W^{*}=\de(w).
\end{equation}
So  we can extend $\de$ to a homomorphism, also denoted 
$\de:B(\ell^{2}(G,\mcH))\to B(\ell^{2}(G\x G,\mcH))$, by defining
\begin{equation}  \label{eq:det} 
\de(T)=W(T\otimes 1)W^{*}.  
\end{equation}
We want to extend it to certain unbounded operators associated with 
$\wht{D}$, specifically, the unbounded operators $D\otimes 1$ and 
$1\otimes M_{c}$ on $B(\ell^{2}(G,\mcH))$.  To this end, let $Z=\mcE\odot C_{c}(G)$.   Then  $Z$ is a dense subspace of 
$\ell^{2}(G\x G,\mcH)$ that is invariant under both $W, W^{*}$.  Also, $Z$ is invariant for $\wht{\pi}(F)\otimes 1, D\otimes 1\otimes 1$ 
and $1\otimes M_{c}\otimes 1$ because of the corresponding properties for $\wht{\pi}(F), D\otimes 1, 1\otimes M_{c}$ for $\mcE$ (in the proof of 
Proposition~\ref{prop:Yst}).   

We now claim that on $Z$ and conjugating with $W$ as in  (\ref{eq:pilade}) to define $\de$ on $D\otimes 1, 1\otimes M_{c}$,
\begin{equation}  \label{eq:d1mc} 
\de(D\otimes 1)=D\otimes 1\otimes 1,\;
\de(1\otimes M_{c})=1\otimes M_{c}\otimes 1.  
\end{equation}
These follow since, for a simple tensor $\ze=h\otimes \xi$, where $h\in Dom\,D$ and $\xi\in C_{c}(G\x G)\in Z$, 
\begin{gather*}
(W(D\otimes 1\otimes 1)W^{*}\ze)(s,t)=((D\otimes 1\otimes 1)W^{*}\eta)(s,s^{-1}t)\\
=D(h)(W^{*}\eta)(s,s^{-1}t)=((D\otimes 1\otimes 1)\ze)(s,t),  
\end{gather*}
and
\begin{gather*} 
(W(1\otimes M_{c}\otimes 1)W^{*}\ze)(s,t)=((1\otimes M_{c}\otimes 1)W^{*}\ze)(s,s^{-1}t)  \\
=hc(s)(W^{*}F)(s,s^{-1}t)=((1\otimes M_{c}\otimes 1)\ze)(s,t).  
\end{gather*}

We will use use the notation $\wht{D}_{\pm}$ for $D\otimes 1 \pm \imath 1\otimes M_{c}$.   To prove the contractive property for $Y$, we recall that for $F\in \mcC$, the operator matrices 
$[\wht{D},(\bt_{\phi}\oplus \bt_{\phi})(F\oplus F)]$ are off-diagonal, and considering their entries, it is sufficient to prove that 
\begin{equation}  \label{eq:phiineq}
\norm{[\wht{D}_{\pm},\bt_{\phi}F]}
\leq \norm{[\wht{D}_{\pm},F]}.  
\end{equation}
and establish the continuity of the maps $\phi\to [\wht{D}_{\pm},\bt_{\phi}F]$.      

From 
(\ref{eq:d1mc}) and (\ref{eq:det}), 
\begin{multline*}
\de([\wht{D}_{\pm}, F])=[\de(D\otimes 1) \pm \imath \de(1\otimes M_{c}),\de(F)] \\
=[(D\otimes 1 \pm \imath 1\otimes M_{c}) \otimes 1,\int \wht{\pi}(F(s))\wht{\la_{s}}\otimes \la_{s}]   
=\int ( [\wht{D}_{\pm},\wht{\pi}(F(s))\wht{\la_{s}})]\otimes \la_{s}. \\ 
\end{multline*}
Of particular significance, this gives that $\de([\wht{D}_{\pm}, F])$ belongs to \\
$B(\mcH\otimes \ell^{2}(G))\otimes C_{r}^{*}(G)$ and we can then use slice maps.   Precisely, if $\phi\in P$, then
\begin{multline*} 
 S_{\phi}(\de([\wht{D}_{\pm},F])) = S_{\phi}(\int ( [\wht{D}_{\pm},\wht{\pi}(F(s))\wht{\la_{s}})]\otimes \la_{s}) \\
=\int ( [\wht{D}_{\pm},\wht{\pi}(F(s))\wht{\la_{s}})]\phi(s) =\int ( [\wht{D}_{\pm},\wht{\pi}(\phi(s)F(s))\wht{\la_{s}})] 
= [\wht{D}_{\pm},\bt_{\phi}F].     
\end{multline*}
The continuity of the maps $\phi\to [\wht{D}_{\pm},\bt_{\phi}F]$ now follows from the corresponding continuity property for slice maps.   
(\ref{eq:phiineq}) also follows using $\norm{S_{\phi}}\leq \norm{\phi}=1$ and the fact that $\de$ is a homomorphism (and so norm decreasing).
\end{proof}

\noindent \textbf{Note on the abelian case}\\
We now discuss how the theorem above simplifies when $G$ is abelian.    The case where $G=\Z$ was examined in detail in \cite[Theorem 2]{BMR}, which relates equicontinuity for $X$ to the isometric condition for $Y$.      Suppose that $X$ is pointwise bounded.  Then we know that $Y$ is contractive.   Let $\chi\in \wht{G}\subset P$.    Then for $F\in \mcC$,
\[  \norm{[\wht{D}_{\pm},F]}=\norm{[\wht{D}_{\pm},\bt_{\chi^{-1}}\bt_{\chi}F]}\leq \norm{[\wht{D}_{\pm},\bt_{\chi}F]}\leq \norm{[\wht{D}_{\pm},F]}  \]
so that $Y$ is isometric, at least with respect to $F\in \mcC$.    However,  because $G$ is abelian, there is a nice formula for $\wht{\pi}(\bt_{\chi}F)$.   As is easily proved (and well-known) 
\[  \wht{\pi}(\bt_{\chi}F)=(1\otimes M_{\chi})\wht{\pi}(F) (1\otimes M_{\chi})^{-1}          \]
where $M_{\chi}$ is the unitary on $\ell^{2}(G)$ given by: $f\to \chi f$ (pointwise multiplication).   It is left to the reader to check that for all $b\in \mcC^{1}(Y)$,
$[\wht{D}_{\pm},\bt_{\chi}(b)]=(1\otimes M_{\chi})[\wht{D}_{\pm},b](1\otimes M_{\chi})^{-1}$ and that we get the isometry condition for all 
$b\in \mcC^{1}(Y)$.    This generalizes part of \cite[Theorem 2]{BMR}, extending from $\Z$ to general abelian $G$ and using the weaker pointwise boundedness condition in place of equicontinuity.   Incidentally, going in the other direction, the isometry condition for $\wht{G}$ gives the contractive property for $P$.   Indeed, contractivity for $\phi\in co\,\wht{G}\subset P$ follows trivially, and by weak$^{*}$-norm continuity, the contractive inequality follows for all 
$\phi\in P\, (=\ov{co\,\wht{G}})$.  

Lastly, from the above,  in the abelian case, a stronger version of Theorem~\ref{th:main} holds, in which the contractive condition holds for all $b\in \mcC^{1}(Y)$
and not just for $b\in \mcC$ as in Theorem~\ref{th:main}.    I do not know if this stronger version also holds for the non-abelian case.

\appendix
\section{Unbounded operators with compact resolvent}

We briefly recall some basic information about unbounded operators on a Hilbert space (e.g. \cite[2.7, 5.6]{KR1}, \cite[p.836f.]{KR2}, \cite{Kato}, \cite[Chapter 13]{Rudin}, 
\cite{DunSchw}.)     Let $D$ be an unbounded linear operator on a Hilbert space $\mcH$ with domain $Dom\, D$.  The operator 
$D$ is called {\em closed} if its graph $\mcG(D)$ is closed in $\mcH\x \mcH$.     It is called {\em preclosed} if the closure of $\mcG(D)$ is itself the graph of a linear operator $\ov{D}$.  In particular, in that case, $\ov{D}$ is  a closed operator, the minimal closed operator that restricts to 
$D$.   If $D$ is closed, a subspace $\mcE$ of $Dom\,D$ is called a {\em core} for $D$ if the graph of $D$ restricted to $\mcE$ is dense in $\mcG(D)$.   (In particular, $\mcE$ is dense in $Dom\, D$.)  We will require the following simple and (no doubt) well known result; for lack of a reference we give the proof.

\begin{proposition}    \label{prop:coredom}
Let $D$ be a closed operator on the Hilbert space $\mcH$.   Let $\mcE$ be a core for $D$ and $D_{\mcE}$ be the restriction of $D$ to 
$\mcE$.   Suppose that $T\in B(\mcH)$ is such that 
$T\mcE\subset \mcE$ and the commutator operator $[D_{\mcE},T]$ on $\mcE$ is bounded.   Let $[D,T]$ be the continuous extension of $[D_{\mcE},T]$
to $Dom\,D$.   Then $T(Dom\,D)\subset Dom\,D$ and the commutator
$[D,T]$ on $Dom\, D$ is bounded.  
\end{proposition}
\begin{proof}
Let $\xi\in Dom\,D$.   Since $\mcE$ is a core for $D$, there exists a sequence $\{\xi_{n}\}$ in $\mcE$ such that 
$\xi_{n}\to \xi, D(\xi_{n})\to D\xi$.    Then $T\xi_{n}\in \mcE$, $T\xi_{n}\to T\xi$ and 
$D(T\xi_{n})=TD\xi_{n} + [D,T]\xi_{n}\to TD\xi + [D,T]\xi$.    So
$T\xi\in Dom\,D$.
\end{proof} 

Now let $D$ have dense domain.   Its {\em adjoint} $D^{*}$ has as its domain the set of $\eta\in \mcH$ for which there is 
a $\zeta$ (which will be unique) such that for all $\xi\in Dom\,D$,
$\lan D\xi,\eta\ran=\lan \xi,\zeta\ran$, and for such an $\eta$, $D^{*}\eta$ is defined to be $\zeta$. 
The unbounded operator $D^{*}$ is always closed, and $D$ is called {\em self-adjoint} if $D=D^{*}$.  In particular, such a 
$D$ is closed. If $D$ is self-adjoint, then (e.g. \cite[Theorem 13.13]{Rudin},  \cite[Remark 2.7.11]{KR1}) $(D\pm \imath 1)$ is a one-to-one map from $Dom\, D$ onto $\mcH$, and its inverse $(D\pm\imath I)^{-1}$ is bounded.  

We will have to consider tensor products of unbounded operators.   
Let $D_{1}, \ldots ,D_{n}$ be densely defined closed operators on Hilbert spaces $\mcH_{1}, \ldots ,\mcH_{n}$.  Then the tensor product 
$D_{1}\odot \cdots \odot D_{n}$ is defined in the obvious way on the algebraic tensor product 
$Dom\,D_{1}\odot \cdots \odot Dom\,D_{n}$.       This operator is preclosed, and  its  closure is denoted by 
$D_{1}\otimes \cdots \otimes D_{n}$.  The algebraic tensor product of cores for the $D_{i}$ is a core for $D$ (\cite[Lemma 11.2.29]{KR2}).
If the $D_{i}$'s are self-adjoint then (\cite[Proposition 11.2.33]{KR2})   
$D_{1}\otimes \cdots \otimes D_{n}$ is also self-adjoint.

We  next describe some of the basic properties of a self-adjoint unbounded operator $D$ on a Hilbert space $\mcH$ {\em with compact resolvent} (\cite{Kato}).   Having a compact resolvent means that for some $\zeta\in \C$, the map $(D - \zeta): Dom(D)\to \mcH$ is one to one and onto, and the resolvent $R(\zeta)=(D - \zeta)^{-1}:\mcH\to Dom\,D\subset \mcH$ is a {\em compact} linear operator.   Then (\cite[p.187]{Kato}) since $D$ is closed (as it is self-adjoint), the compact resolvent property ensures the remarkable facts that the entire spectrum of $D$ consists of isolated eigenvalues 
$\{\la_{k}\}$ with finite-dimensional eigenspaces $E_{k}$, and for every complex number $\la$ which is not an eigenvalue of $D$, $R(\la)$ is compact.  Further (\cite[p.272]{Kato}) all the $\la_{k}$'s are real,  and (\cite[p.277]{Kato}) for $\zeta$ not in the spectrum of $D$, the eigenvalues of $R(\zeta)$ are of the form $(\la_{k} - \zeta)^{-1}$ and have the same set of mutually orthogonal eigenspaces $E_{k}$ and eigenprojections $P_{k}$ as $D$.   Further,  $R(\zeta)=\sum_{k} (\la_{k} - \zeta)^{-1}P_{k}$ in the norm topology.    
In particular, since by compactness, $(\la_{k} - \zeta)^{-1}\to 0$, we have $\left|\la_{k}\right|\to \infty$.    By the spectral theorem for self-adjoint compact operators,  $\sum_{k} P_{k}=1$ in the strong operator topology.   From these facts we  can determine $Dom\, D$.   In fact, $Dom\, D$ is the subspace of all vectors $\xi$ of the form  $\sum_{k}\xi_{k}$ where $\xi_{k}=P_{k}\xi \in E_{k}$ and $\sum_{k}\norm{\xi_{k}}^{2}<\infty, \sum_{k}\la_{k}^{2}\norm{\xi_{k}}^{2}<\infty$, and for such an $\xi$, $D(\sum_{k}\xi_{k}) = \sum_{k}\la_{k}\xi_{k}$.   So we can write $D=\sum_{k}\la_{k}P_{k}$ on $Dom\,D$, convergence being in the strong operator topology.  Conversely given real $\la_{k}$ with $\left| \la_{k}\right|\to \infty$ and a family $E_{k}$ of mutually orthogonal finite dimensional subspaces of $\mcH$ with associated orthogonal projections $P_{k}$ and $\sum_{k} P_{k}=1$, then 
$D=\sum_{k} \la_{k}P_{k}$ defines a self-adjoint operator on $\mcH$ with compact resolvent.      To show this, it is obvious that $D$ is densely defined, and it is simple to check from the definition that if $\eta\in Dom\, D^{*}$, then  $\eta\in Dom\, D$ and $D$ is self-adjoint.   Using the facts that $\left| \la_{k}\right|\to \infty$ and that the $P_{k}$ form a complete orthonormal set of projections, one shows that $(D - \imath)^{-1}$ is a compact normal operator.   So $D$ is self-adjoint with compact resolvent as claimed.   

The following proposition gives the information that we will need about the operator $\wht{D}$ used in this paper.    (We note, by the way, that in the general Hilbert $C^{*}$-module context, Kaad and Lesch (\cite{KaadLesch,KaadLeschspec}) give general conditions that ensure self-adjointness and regularity for a class of two-by-two matrix operators that include $\wht{D}$ below.)   

\begin{proposition}    \label{prop:Dsa}
Let $D_{1}, D_{2}$ be self-adjoint unbounded operators with compact resolvents on the Hilbert spaces $\mcH_{1}, \mcH_{2}$, $\mcK=\mcH_{1}\otimes \mcH_{2}$
and  (as above) $D_{1}\otimes 1, 1\otimes D_{2}$, be the closures of the operators $D_{1}\odot 1, 1\odot D_{2}$.     Define operators $D', \wht{D}$ 
on $\mcK\otimes \C^{2}=\mcK^{2}$ by:
\begin{equation}   \label{eq:DiD1}
  D'=
\begin{bmatrix}
0 & D_{1}\odot 1 - \imath 1\odot D_{2}\\
D_{1}\odot 1 + \imath 1\odot D_{2} & 0
\end{bmatrix}
\end{equation}
and 
\begin{equation}   \label{eq:DiD2}
  \wht{D}=
\begin{bmatrix}
0 & D_{1}\otimes 1- \imath 1\otimes D_{2} \\
D_{1}\otimes 1 + \imath  1\otimes D_{2} & 0
\end{bmatrix}.
\end{equation}
Then $\wht{D}$ is a self-adjoint unbounded operator on $\mcK^{2}$ and is the closure of $D'$.     Further, if $\mcE_{1}, 
\mcE_{2}$ are cores for $D_{1}, D_{2}$, then $\mcE^{2}$, where $\mcE=\mcE_{1}\odot \mcE_{2}$, is a core for $\wht{D}$.
\end{proposition}
\begin{proof}
By the preceding, the operators  $D_{1}\otimes 1$ and $1\otimes D_{2}$ are self-adjoint and $V=Dom\,D_{1}\odot Dom\,D_{2}$ is a core for both.   Since
 \[    V=(Dom\,D_{1}\odot \mcH_{2})\cap (\mcH_{1}\odot Dom\,D_{2})=Dom\,(D_{1}\odot 1)\cap Dom\, (1\odot D_{2}),    \]
 it follows that $Dom\, D'=V^{2}=V\oplus V$.      We now adapt the approach of \cite[p.16]{BMR}.  

In the above notation, we can write $D_{1}=\sum_{k} \la_{k}P_{k}$, $D_{2}=\sum_{r}\mu_{r}Q_{r}$, where the eigenspaces for $D_{1}, D_{2}$ associated with
 $\la_{k}, \mu_{r}$ are $E_{k}, F_{r}$.   Of course, these are also the ranges of the projections $P_{k}, Q_{r}$.    Let
$E_{k,r}=E_{k}\otimes F_{r}$.    Then $E_{k,r}^{2}$ is an eigenspace for the operator $D'$, and the restriction 
$D'_{k,r}$ of $D'$ to $E_{k,r}^{2}$ is the $2\x 2$ matrix
$\left(
\begin{smallmatrix}
0 & (\la_{k} - \imath \mu_{r} )I \\
(\la_{k} + \imath \mu_{r})I & 0 
\end{smallmatrix} 
\right)$
where $I$ is the identity operator on $E_{k,r}$.  An elementary calculation shows that the eigenvalues of $D'_{k,r}$ are 
$\pm\sqrt{\la_{k}^{2} + \mu_{r}^{2}}$.  Let $\la$ be any one of these eigenvalues, and suppose that $\la\ne 0$.  Then the eigenspace for $\la$ is 
\[   E^{\la}_{k,r}=\{(\xi,\eta)'\in E_{k,r}^{2}: \la \xi=(\la_{k} - \imath \mu_{r})\eta  \}.      \]
Since $D'_{k,r}$ is self-adjoint, $E^{2}_{k,r}=E^{\la}_{k,r}\oplus E^{-\la}_{k,r}$ (orthogonal direct sum).   Let 
$P_{\la,k,r}:\mcK^{2}\to E^{\la}_{k,r}$ be the orthogonal projection.    So 
$(P_{k}\otimes Q_{r})\otimes 1=P_{\la,k,r}\oplus P_{-\la,k,r}$.   If $\la=0$, then $D'_{k,r}=0$, and trivially $E^{\la}_{k,r}=E_{k,r}^{2}$ and 
$(P_{k}\otimes Q_{r})\otimes 1=P_{\la,k,r}$.   Then $\{P_{\la,k,r}\}$ ($\la^{2}=\la_{k}^{2} + \mu_{r}^{2}$) is a complete orthonormal family of projections on 
$\mcK^{2}=\oplus _{\la, k, r} (E^{\la}_{k,r})^{2}$ (since $\{(P_{k}\otimes Q_{r})\otimes 1\}$ is) and $\left| \la \right|\to \infty$ as 
$k^{2} + r^{2}\to \infty$.   

Let $L$ be the self-adjoint operator with compact resolvent associated above with the $\la$'s and $P_{\la,k,r}$: so $Dom\, L$ is the space of $[\xi,\eta]'=\{[\xi_{\la,k,r},\eta_{\la,k,r}]'\}$ in $\mcK^{2}$ for which $\sum \la [\xi_{\la,k,r},\eta_{\la,k,r}]'\in 
\mcK^{2}$, and for such an
$[\xi,\eta]'$, $L[\xi,\eta]'=\sum \la [\xi_{\la,k,r},\eta_{\la,k,r}]'$.  Let $W$ be the space of $\xi\in Dom\, L$ for which $\xi_{\la,k,r}=0$ except for a finite number of triples $(\la,k,r)$.    (So $W$ is just the linear span of $\cup_{k,r} E_{k,r}^{2}$ in $\mcK^{2}$.)   
It is left to the reader to check that $L,D',\wht{D}$ coincide on $W$, and that $W$ is dense in $\mcK^{2}$, and is a core for $L, \wht{D}$ and $D'$.    So the closure of $D'$ is $L$.    It remains to show that $\wht{D}=L$.  

To this end, we first determine the domain of $\wht{D}$.   First, a core for $D_{1}\odot 1$ is the space of all linear combinations of elements of the form 
$\xi_{k,r}\in E_{k,r}$ over $k,r$.       Since by definition, $D_{1}\otimes 1$ is the closure of $D_{1}\odot 1$, its domain is the space of 
elements $\xi\in \mcK$ such that $\sum \la_{k}\xi_{k,r}\in \mcK$ and $(D_{1}\otimes 1)(\xi)=\sum \la_{k}\xi_{k,r}$.  Similarly, the domain of 
$1\otimes D_{2}$ is the space of elements $\eta\in \mcK$ such that $\sum \mu_{r}\eta_{k,r}\in \mcK$ and 
$(1\otimes D_{2})(\eta)=\sum \mu_{r}\eta_{k,r}$.  Hence the domain of $D_{1}\otimes 1 \mp \imath 1\otimes D_{2}$ is the space 
\begin{equation}   \label{eq:whtV}
\wht{V}=\{\xi\in \mcK:  \mbox{  both } \sum \la_{k}\xi_{k,r}, \sum \mu_{r}\xi_{k,r}\in \mcK\}. 
\end{equation}  
Obviously, if $\xi\in \mcK$, then $\xi\in \wht{V}$ if and only if $\sum (\la_{k} \mp \imath \mu_{r})\xi_{k,r}\in \mcK$, since that amounts to saying that
$\sum (\la_{k}^{2} + \mu_{r}^{2})\norm{\xi_{k,r}}^{2}<\infty$.   The domain of $\wht{D}$ is then $\wht{V}^{2}$, and this is the same as $Dom\, L$.   Indeed, 
$\sum_{k,r} (\la_{k}^{2} + \mu_{r}^{2})\norm{[\xi_{k,r},\eta_{k,r}]'}^{2}=
\sum_{k,r} (\la_{k}^{2} + \mu_{r}^{2})(\norm{\xi_{k,r}}^{2} + \norm{\eta_{k,r}}^{2}) = \sum \la^{2}\norm{[\xi_{\la,k,r},\eta_{\la,k,r}]'}^{2}$.
Since both $\wht{D}, L$ coincide on every $E_{k,r}^{2}$, they are the same on their domain $\wht{V}^{2}$.

Now let $\mcE_{i}$ be cores for $D_{i}$ and $\mcE=\mcE_{1}\odot \mcE_{2}$.   Then trivially, $\mcE^{2}\subset Dom\, \wht{D}$.  
Since $\wht{V}^{2}$ is the domain of $\wht{D}$, we just have to show that each pair
$(\zeta,D'\zeta)$, where $\zeta=[\xi_{k}\otimes \eta_{r}, \xi'_{k}\otimes \eta'_{r}]'$ with  
$\xi_{k}, \xi'_{k}\in E_{k}, \eta_{r}, \eta'_{r}\in F_{r}$, is in the closure of the graph of $\wht{D}$ restricted to $\mcE^{2}$.   To prove this, we need only show that 
$(\xi_{k}\otimes \eta_{r},D_{1}\xi_{k}\otimes \eta_{r} \mp \imath\xi_{k}\otimes D_{2}\eta_{r})$ is in the closure of the graph of $D_{1}\odot 1 \mp \imath 1\odot D_{2}$ restricted to $\mcE$.    This follows since there are sequences $\{v_{n}\}, \{w_{n}\}$ in $\mcE_{1}, 
\mcE_{2}$ such that $(v_{n},D_{1}v_{n})\to (\xi_{k},D_{1}\xi_{k})=(\xi_{k},\la_{k}\xi_{k})$ and 
$(w_{n},D_{2}w_{n})\to (\eta_{r},D_{2}\eta_{r})=(\eta_{r},\mu_{r}\eta_{r})$.
\end{proof}

\end{document}